\documentclass[11pt]{amsart}
\usepackage{amsfonts}
\usepackage{graphicx}
\usepackage{tabularx}
\usepackage{array}
\usepackage[usenames,dvipsnames]{color}
\usepackage{comment}
\usepackage{amsmath}
\usepackage{amsthm}
\usepackage{amssymb}
\usepackage{fullpage}
\usepackage{xcolor}
\usepackage{tikz}
\usepackage{listings}

\tikzstyle{legend_general}=[rectangle, rounded corners, thin,
       top color= white,bottom color=lavander!25,
       minimum width=2.5cm, minimum height=0.8cm,
       violet]

\renewcommand{\epsilon}{\varepsilon}

\DeclareMathOperator{\diam}{diam}
\DeclareMathOperator{\dist}{dist}
\DeclareMathOperator{\E}{E}

\newtheorem{theorem}{Theorem}[section]

\newtheorem{claim}[theorem]{Claim}
\newtheorem{lemma}[theorem]{Lemma}
\newtheorem{corollary}[theorem]{Corollary}
\theoremstyle{definition}
\newtheorem{definition}[theorem]{Definition}

\author{Peter Bradshaw}
\address{Department of Mathematics, Simon Fraser University, Burnaby, BC, Canada}
\email{pabradsh@sfu.ca}

\author{Bojan Mohar}
\address{Department of Mathematics, Simon Fraser University, Burnaby, BC, Canada}
\email{mohar@sfu.ca}

\title{A rainbow connectivity threshold for random graph families}

\begin{document}

\begin{abstract}
Given a family $\mathcal G$ of graphs on a common vertex set $X$, we say that $\mathcal G$ is \emph{rainbow connected} if for every vertex pair $u,v \in X$, there exists a path from $u$ to $v$ that uses at most one edge from each graph in $\mathcal G$. We consider the case that $\mathcal G$ contains $s$ graphs, each sampled randomly from $G(n,p)$, with $n = |X|$ and $p = \frac{c \log n}{sn}$, where $c > 1$ is a constant. We show that when $s$ is sufficiently large, $\mathcal G$ is a.a.s.~rainbow connected, and when $s$ is sufficiently small, $\mathcal G$ is a.a.s.~not rainbow connected. We also calculate a threshold of $s$ for the rainbow connectivity of $\mathcal G$, and we show that this threshold is concentrated on at most three values, which are larger than the diameter of the union of $\mathcal G$ by about $\frac{\log n}{(\log \log n)^2}$.

The same results also hold in a more traditional random rainbow setting, where we take a random graph $G\in G(n,p)$ with $p=\frac{c \log n}{n}$ ($c>1$) and color each edge of $G$ with a color chosen uniformly at random from the set $[s]$ of $s$ colors.
\end{abstract}

\maketitle

\section{Introduction}

In this paper, we consider random graphs using the \emph{Erd\H{o}s-R\'enyi model}, which are defined as follows. For a positive integer $n$, we consider a set $X$ of $n$ vertices. Then, for some value $0 \leq p \leq 1$, we construct a graph $G$ on $X$ by independently letting each edge $e \in \binom{X}{2}$ belong to $E(G)$ with probability $p$. We say that $G$ is a \emph{random graph} in $G(n,p)$. When a statement involving a value $n$ holds with probability approaching $1$ as $n$ approaches infinity, we say that the statement holds \emph{asymptotically almost surely}, or \emph{a.a.s.}~for short.

\subsection{Background}
One particular property of random graphs that has been the focus of extensive research is the diameter. Recall that the diameter $\diam(G)$ of a graph $G$ is defined as the maximum distance $\dist(u,v)$ taken over all vertex pairs $u,v$ in the graph.
Random graphs are often used as theoretical models for complex networks \cite{Network1} \cite{Network2}, 
in which case the diameter of a random graph represents the maximum degree of separation between any two nodes in a network. Therefore, the diameter of random graphs is often studied in order to gain a better understanding of the connections between elements of real systems.
In a seminal paper on random graphs from 1959, Erd\H{o}s and R\'enyi \cite{Erdos1959} showed that if a graph $G$ is randomly sampled from $G(n,p)$ with $p = \frac{c \log n}{n}$ and $c$ a constant, then $G$ is a.a.s.~connected when $c > 1$ and a.a.s.~disconnected when $c < 1$. This result of Erd\H{o}s and R\'enyi was essentially the first result on the diameter of random graphs, 
giving a probability threshold for when the diameter of a random graph is finite. Later, in 1974, Burtin \cite{Burtin} determined that when $p \gg n^{-\frac{d-1}{d}}$ for a positive integer $d$, a random graph sampled from $G(n,p)$ a.a.s.~has a diameter of at most $d$, and Klee and Larman \cite{Klee} rediscovered this result independently in 1981. Bollob\'as \cite{Bollobas84} then showed in 1984 that when
a graph $G$ on $n$ vertices has $\frac{c \log n}{n} \binom{n}{2}$ randomly placed edges (where $c > 1$ is a constant), the graph $G$ has a diameter that is a.a.s.~equal to one of at most four consecutive integer values. Chung and Lu \cite{Chung} later translated this result of Bollob\'as into the random setting $G(n,p)$, giving the following bounds:
\begin{theorem} \cite{Chung}
\label{thmChung}
Let $G$ be a random graph in $G(n,p)$, where $p = \frac{c \log n}{n}$ and $c > 1$. Then a.a.s.,
$$\frac{\log \left( \frac{c}{11} \right) + \log n}{\log c + \log \log n} \leq \diam(G) \leq \frac{\log \left( \frac{33c^2}{400} \right) + \log \log n + \log n}{\log c + \log \log n} + 2.$$
\end{theorem}
In other words, Chung and Lu show that the diameter of $G$ is a.a.s.~one of at most four consecutive integer values, each within a constant from $\frac{\log n}{\log c + \log \log n}$.

In seeking the diameter of a random graph $G$, one essentially asks the following question: For which values of $s$ does there a.a.s.~exist a path of length at most $s$ between every pair of vertices in $G$? 
In this paper, we ask a similar question in the following \emph{rainbow} setting.
 
We consider a family $\mathcal G = \{G_1, \dots, G_s\}$ of $s$ graphs on a common vertex set $X$ of size $n$. We say that a path $P \subseteq \bigcup_{i = 1}^s E(G_i)$ is a \emph{rainbow path} if there exists an injection $\phi:E(P) \rightarrow [s]$ such that for each edge $e \in E(P)$, $e \in E(G_{\phi(e)})$. 
For two vertices $u,v \in X$, we say that $u$ and $v$ are \emph{rainbow connected} if there exists a rainbow path with $u$ and $v$ as endpoints. Furthermore, we say that $\mathcal G$ is \emph{rainbow connected} if every pair $(u,v) \in \binom{X}{2}$ is rainbow connected. 
If we let each graph $G_i \in \mathcal G$ have its edges colored with the color $i$, then we may equivalently define a rainbow path as a path that uses at most one edge of each color. For a given family of random graphs, we will ask, for which values of $s$ does there a.a.s.~exist a rainbow path of length at most $s$ between every pair of vertices in $X$? Equivalently, we ask, for which values of $s$ is $\mathcal G$ a.a.s.~rainbow connected? 

The notion of rainbow connectivity was first introduced by Chartrand et al.~\cite{Chartrand} as a theoretical tool for studying communication in secure networks. They defined the \emph{rainbow connection number} of a graph $G$ as the minimum number of colors needed in order to give $G$ an edge-coloring with which $G$ is rainbow connected. 
The rainbow connection number is well-understood for certain graph classes including trees, cycles \cite{Chartrand}, and Cayley graphs on abelian groups \cite{LiRainbow}. For random graphs $G$ sampled from $G(n,p)$ with $p = \frac{\log n + \omega}{n}$ and $\omega = o(\log n)$ an unbounded increasing function, Frieze and Tsourakakis \cite{Frieze} showed that the rainbow connection number of $G$ asymptotically approaches the diameter of $G$.

Rainbow paths are an example of rainbow graphs, which are defined as edge-colored graphs in which each edge has a unique color. Depending on the setting, rainbow graphs are also often referred to either as transversals or partial transversals of a graph family. Recently, rainbow graph structures have received increasing attention, and several classical results have been extended into the rainbow setting. For instance, a famous theorem of Dirac \cite{Dirac} states that a graph on $n$ vertices with minimum degree at least $n/2$ must contain a Hamiltonian cycle. Joos and Kim \cite{Joos} have generalized Dirac's result to show that given a family $\mathcal G = \{G_1, \dots, G_n\}$ of $n$ graphs on a common set $X$ of $n$ vertices, each of minimum degree at least $n/2$, where the edges of each graph $G_i$ are monochromatically colored with the color $i$, there must exist a Hamiltonian cycle on $X$ using exactly one edge of each color. In the same flavor, a classic result of Moon and Moser \cite{Moon} gives a minimum degree condition for the existence of a Hamiltonian cycle in a bipartite graph, and one of the authors of this paper has recently shown a similar generalization of this result into the rainbow setting \cite{PBbipartite}. In addition to rainbow Hamiltonian cycles, certain other rainbow structures have been shown to exist under appropriate conditions. For instance, Aharoni et al.~\cite{Aharoni} obtained a rainbow version of Mantel's theorem, proving that given a family $\mathcal G$ of three graphs on a common set of $n$ vertices, if each graph in $\mathcal G$ contains at least $0.2557n^2$ edges, then there exists a rainbow triangle---that is, a triangle that uses exactly one edge from each graph of $\mathcal G$.

Above, we introduced rainbow connectivity in a graph family $\mathcal G$ that is the union of many individual random graphs on a common vertex set, each with monochromatically colored edges of a distinct color. There is also a more traditional setting for discussing randomly edge-colored graphs, in which a single graph $G$ is constructed using some random process, and then each edge of $G$ is randomly given a color. For example, Frieze and McKay \cite{FriezeMcKay} show that when $G$ is constructed by randomly adding edges one at a time and giving each new edge one of $n-1$ colors uniformly at random, the time at which $G$ first contains edges of every color almost surely coincides with the time at which $G$ first contains a rainbow spanning tree. Additionally, Ferber and Krivelevich \cite{Ferber} show that for a graph $G$ randomly sampled from $G(n,p)$ with $p = \frac{\log n + \log \log n + \omega}{n}$, with $\omega$ an unbounded increasing function, if each edge of $G$ randomly receives one of $(1+o(1))n$ colors, then $G$ a.a.s.~contains a rainbow Hamiltonian cycle.

\subsection{Our results}
First, we fix some notation that we will use throughout the paper. We let $X$ be a set of $n$ vertices, where $n$ is a large integer. We pick an integer $s \geq 1$, depending on $n$, and we let $c > 1$ be a fixed constant. We let 
$$p = \frac{c \log n}{sn}.$$
Then, for $1 \leq i \leq s$, we take a random graph $G_i \in G(n,p)$, and we let $\mathcal G= \{G_1, \dots, G_s \}$. We often refer to the values $1, \dots, s$ as \emph{colors}, and we will 
imagine
that each graph $G_i$ has its edges colored with color $i$.
In this setting, it is straightforward to show that an edge $e \in \binom{X}{2}$ belongs to at least one graph $G_i \in \mathcal G$ with probability $(c - o(1))\frac{\log n}{n}$, and therefore, by the threshold of Erd\H{o}s and R\'enyi \cite{Erdos1959}, $\bigcup_{i = 1}^s G_i$ is a.a.s.~connected.

In the following two main results, we determine a threshold for the number $s$ of graphs required in $\mathcal G$ to ensure rainbow connectivity.

\begin{theorem}
\label{thmUB}
Let $\mathcal G = \{G_1, \dots, G_s\}$ be a family of $s$ graphs on a common set of $n$ vertices, each taken randomly from $G(n,p)$, with $p = \frac{c \log n}{sn}$, where $c > 1$ is a constant.
If $$s \leq \frac{\log n}{\log c - 1 + \log \log n} - \frac{1}{2} + \frac{\log \log \log n}{3 \log \log n},$$
then a.a.s.~$\mathcal G$ is not rainbow connected. 
\end{theorem}

\begin{theorem}
\label{thmEasyLB}
Let $\mathcal G = \{G_1, \dots, G_s\}$ be a family of $s$ graphs on a common set of $n$ vertices, each taken randomly from $G(n,p)$, with $p = \frac{c \log n}{sn}$, where $c > 1$ is a constant.
If 
$$s \geq \frac{\log n}{\log c - 1 + \log \log n} + \frac{3}{2} + \frac{2 \sqrt{ \log \log \log n }}{ {\log \log n}} ,$$
then a.a.s.~$\mathcal G$ is rainbow connected. 
\end{theorem}

The minimum value $s$ that guarantees that the $s$ graphs in $\mathcal G$ a.a.s.~make a rainbow connected family is called the \emph{rainbow connectivity threshold}.
Together, Theorems \ref{thmUB} and \ref{thmEasyLB} show that this threshold is concentrated on at most three consecutive integer values 
in the vicinity of $ \frac{\log n}{\log c - 1 + \log \log n}$: two integer values that may fit between the two bounds, along with the smallest integer greater than or equal to the bound of Theorem \ref{thmEasyLB}. 

\begin{corollary}
The rainbow connectivity threshold is a.a.s.~equal to one of the three values, $s_0,\allowbreak s_0+1,s_0+2$, where $s_0 = \bigl\lfloor \frac{\log n}{\log c - 1 + \log \log n} + \frac{1}{2} + \frac{\log \log \log n}{3 \log \log n} \bigr\rfloor$.
\end{corollary}

It is simple to observe that in order for $\mathcal G$ to be rainbow connected, $s$ must be at least as large as the diameter of $\bigcup_{i = 1}^s G_i$. However, by comparing the threshold found in Theorems \ref{thmUB} and \ref{thmEasyLB} with the result of Theorem \ref{thmChung}, we see in fact that letting $s$ equal the diameter of $\bigcup_{i = 1}^s G_i$ is not enough to ensure a.a.s. the rainbow connectivity of $\mathcal G$.
Indeed, according to Theorem \ref{thmChung}, the diameter of $\bigcup_{i = 1}^s G_i$ is a.a.s.\ close to $\frac{\log n}{\log c + \log \log n}$, which is slightly smaller (by about $\frac{\log n}{(\log \log n)^2}$) than our rainbow connectivity threshold.

Finally, in our concluding section we discuss rainbow connectivity in the more traditional random rainbow setting, where we take a random graph $G\in G(n,p)$ with $p=\frac{c \log n}{n}$ ($c>1$) and color each edge of $G$ with a color chosen uniformly at random from the set $[s]$ of $s$ colors. In Section \ref{sec:differentSetting}, we show that results similar to Theorems \ref{thmUB} and \ref{thmEasyLB} also hold in this alternative type of random setting. See Theorems \ref{thm:newUB} and \ref{thm:newLB}.

\section{Proof of Theorem \ref{thmUB}}

We start by proving the simpler one of the two main results.

\begin{proof}[Proof of Theorem \ref{thmUB}]
We show that for an arbitrary vertex pair $u,v \in X$ ($u\ne v$), there a.a.s.~exists no rainbow path from $u$ to $v$. We will use the First Moment Method (cf. \cite[Chapter 3]{MolloyReed}).

Let $u,v \in X$ be a vertex pair, and let $1 \leq t \leq s$. The total number of possible rainbow colored paths of length $t$ from $u$ to $v$ is less than $n^{t-1} s!$, and the probability of any such rainbow path's existence in $\mathcal G$ is equal to $p^t$. Therefore, the expected number of rainbow paths of length $t$ from $u$ to $v$ is at most $n^{t-1} s! \,p^t$. Now, using Stirling's approximation, we may estimate that the expected number of rainbow paths from $u$ to $v$ of length $t$ is at most 
\begin{eqnarray*}
\frac{(pn)^t}{n}\, s! &<& \frac{(pn)^t}{n}\cdot s^{s + \frac{1}{2}} e^{-s+1} \\
 & = & \frac{1}{n} \left( \frac{psn}{e} \right)^t s^{s-t+\frac{1}{2}} e^{t-s+1} \\
 & = & \frac{1}{n} \left( \frac{c \log n}{e} \right)^t \left(\frac{s}{e}\right)^{s-t+\frac{1}{2}} e^{3/2}.
\end{eqnarray*}
Since 
$$s \leq \frac{\log n}{\log c - 1 + \log\log n} - \frac{1}{2} + \frac{\log\log\log n}{3\log\log n},$$
we can write $s = \frac{\log n}{\log c - 1 + \log\log n} +k$, where 
$$k \le - \frac{1}{2} + \frac{\frac{1}{2} \log\log\log n - \sqrt{\log\log\log n} }{\log c - 1 + \log\log n} .$$ 
Note that $n = \left(\frac{c \log n}{e} \right)^{s-k}$.
Then, letting $Y$ be the number of rainbow paths from $u$ to $v$ of any length $t\ge1$, we have
\begin{eqnarray*}
\E[Y] < \sum_{t = 1}^s \left( \frac{c \log n}{e} \right)^{t-s +k } \left( \frac{s}{e} \right)^{s-t+\frac{1}{2}} e^{3/2} & = & \left( \frac{c \log n}{e} \right)^k e \sqrt{s} \cdot \sum_{t = 1}^s \left( \frac{s}{c \log n} \right)^{s-t} .
\end{eqnarray*}
Since $s < \frac{2 \log n}{\log \log n}$, we have
\begin{eqnarray*}
 \E[Y] & < & \left( \frac{c \log n}{e} \right)^k e\, \sqrt{\frac{2\log n}{\log \log n}} \, \sum_{t = 1}^s \left( \frac{2}{c \log \log n} \right)^{s-t} .
 \end{eqnarray*}
Since the sum is less than $2$, and since $\sqrt{2} e < 4$, we have 
 \begin{eqnarray*}
 \E[Y] & < & 8 \left( \frac{c \log n}{e} \right)^k \sqrt{\frac{\log n}{\log \log n}} .
 \end{eqnarray*}
Finally, since $k+\frac{1}{2}$ is at most the logarithm of 
$\sqrt{\log\log n} \exp \bigl(-\sqrt{\log \log \log n}\,\bigr)$ with a base of $c \log n / e$, it follows that 
\begin{eqnarray*}
 \E[Y] & < & 8
 e^{1/2} c^{-1/2} \exp \left( {- \sqrt { \log \log \log n } } \right) \rightarrow 0. 
\end{eqnarray*}
By Markov's inequality, the probability that there exists a rainbow path from $u$ to $v$ is at most $\E[Y]$, and thus it holds that for an arbitrarily chosen vertex pair $u,v \in X$, $u$ and $v$ are a.a.s.~not rainbow connected.
\end{proof}

\section{Tools and key ideas}

In this section, we outline some tools and key ideas that we will need for the proof of Theorem \ref{thmEasyLB}. For this entire section and the next, we set 
\begin{equation}
  d = \lceil \log \log \log \log n \rceil. \label{eq:def_d}
\end{equation}

We will need several inequalities that will help us estimate various probabilities. First, we have the following inequality, which follows easily from the inequality $(1-p)^x \le \exp(-px)$. We will use it throughout the entire paper without explicitly stating that we are doing so.
\begin{lemma}
If\/ $p,x > 0$ and $px < 1$, then $1 - (1-p)^x > px - p^2 x^2.$
\end{lemma}
Next,
we will use the following forms of the Chernoff bound, which can be found, for example, in Chapter 4 of \cite{Mitzenmacher}. 

\begin{theorem}
\label{thmChernoff}
Let $Y$ be a random variable that is the sum of pairwise independent indicator variables---that is, random variables taking values in $\{0,1\}$. Let $\mu = \E[Y]$. Then for any value $\delta \in (0,1)$,
$$\Pr(Y < (1 - \delta) \mu) \leq \left( \frac{e^{-\delta}}{(1-\delta)^{1-\delta}} \right)^{\mu} ,$$
and for any value $\delta > 0$,
$$ \Pr(Y > (1 + \delta) \mu) \leq \left( \frac{e^{\delta}}{(1+\delta)^{1+\delta}} \right)^{\mu}.$$
\end{theorem}

\begin{theorem}
\label{thmSimpleChernoff}
Let $Y$ be a random variable that is the sum of pairwise independent indicator variables, and let $\mu = \E[Y]$. Then for any value $\delta > 0$,
$$\Pr(Y > (1 + \delta) \mu) \leq \exp \left( - \frac{ \delta^2 \mu }{2 + \delta }\right)$$
and for any value $\delta \in (0,1)$,
$$\Pr(Y < (1 - \delta) \mu) \leq \exp \left( - \frac{1}{2} \delta^2 \mu \right).$$
\end{theorem}

Mitzenmacher \cite{Mitzenmacher} points out that for the first statement of each theorem, it is enough to let $\mu \leq \E[Y]$, and for the second statement of each theorem, it is enough to let $\mu \geq \E[Y]$.

\subsection{Bounding the relevant values of $s$}
\label{bddS}

Theorems \ref{thmUB} and \ref{thmEasyLB} show that the rainbow connectivity threshold of $\mathcal G$ is of the form $(1+o(1)) \frac{\log n}{\log \log n}$. Intuitively, this should imply that the most difficult values of $s$ for which to prove these theorems are those close to $\frac{\log n}{\log \log n}$ and that other values of $s$ need not be carefully considered. In the following results, we will formalize this intuition.

First, by Theorem \ref{thmUB}, we may assume that $s \geq \frac{\log n}{2 \log \log n}$.
Next, we show that it is sufficient just to prove Theorem \ref{thmEasyLB} for values $s \leq \frac{3 \log n}{\log \log n}$. The proof of this lemma gives some intuition regarding why $\mathcal G$ is more likely to be rainbow connected for larger values of $s$.

\begin{lemma}
\label{lemmaSmall}
Suppose that whenever
 $$\frac{\log n}{\log c - 1 + \log \log n} + \frac{3}{2} + \frac{2 \sqrt{ \log \log \log n }}{ {\log \log n}} \leq s \leq \frac{3 \log n}{\log \log n},$$
$\mathcal G$ is a.a.s.~rainbow connected. Then $\mathcal G$ is a.a.s.~rainbow connected also for any value $s > \frac{3 \log n}{\log \log n}$.
\end{lemma}

\begin{proof}
Suppose that $s > \frac{3 \log n}{\log \log n}$. Let $k$ be the smallest power of $2$ for which $ \left \lfloor s/k \right \rfloor < \frac{3 \log n}{\log \log n}$. 
By our choice of $k$, we know that 
\begin{equation}
\label{eqnSK}
\frac{s}{k} \geq \frac{3 \log n}{2 \log \log n} .
\end{equation}
We then create a family $\mathcal H$ of $\lfloor s / k \rfloor$ graphs by letting $$H_i = G_{ki+1} \cup G_{ki + 2} \cup \dots \cup G_{k(i+ 1)}$$ for $0 \leq i < \lfloor s / k \rfloor$. We observe that each graph $H_i \in \mathcal H$ is randomly sampled from $G(n,p^*)$, where $p^* = 1-(1-p)^k$. We estimate
$$p^* \left \lfloor \frac{s}{k} \right \rfloor > \frac{s-k}{k} (kp-k^2p^2) > ps \left( 1 - \frac{k}{s} - kp \right) > ps \left( 1 - \frac{2k}{s} \right) = \frac{c^* \log n}{n},$$
where $c^* = c \left( 1 - \frac{2k}{s} \right).$
By our hypothesis, if $\lfloor s / k \rfloor \geq \frac{\log n}{\log c^* - 1 + \log \log n} + \frac{3}{2} + \frac{2 \sqrt{ \log \log \log n }}{ {\log \log n}} $,
 then $\mathcal H$ is a.a.s.~rainbow connected, which implies that $\mathcal G$ is a.a.s.~rainbow connected. 
We also see
$$\log c^* = \log c + \log \left( 1 - \frac{2k}{s} \right) > \log c - \frac{3k}{s}.$$
Therefore, it suffices to show that 
$$\left \lfloor \frac{s}{k} \right \rfloor \geq \frac{\log n}{\log c - \frac{3k}{s} - 1 + \log \log n} + 2 , $$
or stronger still, that
$$\left \lfloor \frac{s}{k} \right \rfloor > (1 + o(1)) \frac{\log n}{\log \log n}$$
holds for every function $o(1)$. However, this follows immediately from (\ref{eqnSK}). 
 Hence, we know that $\mathcal H$ is a.a.s.~rainbow connected, and thus $\mathcal G$ is also a.a.s.~rainbow connected.
\end{proof}

By combining Theorem \ref{thmUB} and Lemma \ref{lemmaSmall} with the hypothesis of Theorem \ref{thmEasyLB}, we assume throughout the rest of the paper that $s$ satisfies
\begin{equation}\tag{$\star$}
\label{sBound}
\frac{\log n}{2 \log \log n} \leq s \leq \frac{3 \log n}{\log \log n}. 
\end{equation}

\subsection{Spheres and breadth-first search for rainbow paths}
\label{BFS}

In this subsection, we will outline a breadth-first search technique that we use extensively in the proof of Theorem \ref{thmEasyLB}.
For an ordinary graph $G$ and a vertex $v \in V(G)$, the \emph{sphere} of radius $t$ around $v$, denoted by $\Gamma_t(v)$, is defined as the set of vertices in $V(G)$ at distance exactly $t$ from $v$. Then, the statement that $\diam(G) > t$ is equivalent to the statement that there exists a vertex pair $u,v \in V(G)$ for which $u \not \in \bigcup_{i = 1}^t \Gamma_i(v)$. For each value $t \geq 0$, $\Gamma_t(v)$ can be calculated by carrying out a breadth-first search on $G$ starting at $v$ and searching up to distance $t$. Therefore, for ordinary graphs, the concepts of a graph's diameter, spheres around vertices, and breadth-first search are intimately related.

In our rainbow setting, we aim to determine the number $s$ of colors needed to make $\mathcal G$ a.a.s.~rainbow connected.
We will see that similarly to a graph's diameter, the value $s$ also has bounds that are closely related to breadth first search and spheres, which brings us to the following definition.

\begin{definition}
Let $v \in X$, $C \subseteq [s]$, and let $t \geq 0$ be an integer. Then we define $\Gamma_t^C(v)$ to be the set of vertices $u \in X$ satisfying the following two conditions:
\begin{itemize}
\item[(i)] $u$ can be reached from $v$ by a rainbow path of length $t$ consisting of edges of graphs $G_i$ for which $i \in C$;
\item[(ii)] $u$ cannot be reached from $v$ by a rainbow path of length at most $t-1$ consisting of edges of graphs $G_i$ for which $i \in C$.
\end{itemize}
\end{definition}
We will refer to these sets $\Gamma_t^C(v)$ as \emph{spheres}.
We observe that for each vertex $v \in X$, $v$ is rainbow connected with every vertex in $\bigcup_{i = 0}^s\Gamma_i^{[s]}(v)$.
In fact, $v$ would be rainbow connected with every vertex in $\bigcup_{i = 0}^s\Gamma_i^{[s]}(v)$ even if we removed the second condition in the definition of each $\Gamma_i^{[s]}(v)$, but we will see later that we need this second condition for technical reasons.

Similarly to spheres in ordinary graphs, for a vertex $v \in X$, the sets $\Gamma_i^{C}(v)$ can be computed recursively with a breadth-first search.
First, we let $\Gamma_0^{C}(v) = \{v\}$. Then, for $0 \leq t \leq |C| - 1$, we can compute $\Gamma_{t+1}^{C}(v)$ from $\Gamma_{t}^{C}(v)$ as follows. 
We consider each vertex $w \in \Gamma_{t}^{C}(v)$ individually. By definition, there exists a nonempty set $\mathcal P_w$ of rainbow paths from $v$ to $w$ of length exactly $t$, each of which using only edge colors in $C$. Using each path $P \in \mathcal P_w$, we define $C(P)$ as the set of $t$ colors appearing at the edges $E(P)$. Then, for each path $P \in \mathcal P_w$, we search for vertices $u \in X \setminus \bigcup_{i = 0}^t \Gamma^C_i$ for which $wu \in E(G_j)$ for some color $j \in C \setminus C(P)$. Whenever we find such a vertex $u$, we add $u$ to $\Gamma_{t+1}^{C}(v)$. By carrying out this process for each vertex $w \in \Gamma_t^{C}(v)$ and each path $P \in \mathcal P_w$, we determine $\Gamma_{t+1}^{C}(v)$.

When we calculate bounds for $s$, we will be interested in estimating the sizes of the spheres $\Gamma_t^C(v)$. We have several results that will help us.
The first one shows that each individual graph $G_i \in \mathcal G$ almost surely has a small maximum degree.

\begin{lemma}
\label{lemmaDelta}
It holds a.a.s.~that for each graph $G_i \in \mathcal G$, $\Delta(G_i) < \frac{2c \log n}{\log \log n}.$
\end{lemma}

\begin{proof}
Let $v \in X$, and let $G_i \in \mathcal G$. The expected value $\mu$ of $\deg_{G_i}(v)$ is equal to 
$$\mu = p(n-1) < \frac{c \log n}{s} \leq 2c \log \log n =: \nu,$$
using the bound (\ref{sBound}).
Hence, by applying a Chernoff bound (Theorem \ref{thmChernoff} with the remark stated after the theorem),
$$\Pr \left( \deg_{G_i}(v) > \frac{2 c \log n}{\log \log n} \right) < \left( \frac{e^\delta}{(1 + \delta)^{1 + \delta}}\right)^{\nu},$$
where $\delta = \frac{ \log n}{(\log \log n)^2}$.
Calculating further,
\begin{eqnarray*}
 \left( \frac{e^\delta}{(1 + \delta)^{1 + \delta}}\right)^{\nu} &=& \exp( \nu \delta - \nu (1 + \delta) \log (1 + \delta) )\\
 &<& \exp(\nu \delta (1 - \log \delta)) \\
 &=& \exp \left( \frac{2c \log n}{\log \log n} (- \log \log n +2 \log \log \log n + 1) \right)\\
 & =& o \left(\frac{1}{n^2} \right).
\end{eqnarray*}
Therefore, it a.a.s.~holds that for each vertex $v \in X$, and for each graph $G_i \in \mathcal G$, $\deg_{G_i}(v) < \frac{2c \log n}{\log \log n}$.
\end{proof}

Next, we show that for each vertex $v \in X$, there are logarithmically many edges of $\mathcal G$ incident to $v$, even when a set of $d+1$ colors is removed, where $d$ is defined in (\ref{eq:def_d}).

\begin{lemma}
\label{lemFirstStep}
There exists a value $\epsilon = \epsilon(c) > 0$ such that a.a.s., for every vertex $v \in X$ and every set $C \subseteq [s]$ of size at most $d+1$,
$$\left |\Gamma^{[s] \setminus C}_1(v) \right | \geq \epsilon \log n.$$
\end{lemma}

\begin{proof}
Let $v \in X$. We first estimate $\left |\Gamma^{[s]}_1(v) \right |$. The probability that $v$ is adjacent to a given vertex $u \in X \setminus \{v\}$ in at least one graph $G_i \in \mathcal G$ is equal to $1 - (1-p)^{s} > ps - p^2s^2.$ Therefore, the expected number of vertices in $\Gamma^{[s]}_1(v)$ is at least 
$$\nu = (n-1)(ps - p^2s^2) > \alpha \log n,$$
for some constant $\alpha > 1$ depending only on $c$. Therefore, by a Chernoff bound (Theorem \ref{thmChernoff}), for a constant $\beta > 0$, 
$$\Pr \left( \left |\Gamma^{[s]}_1(v) \right | < \beta \nu \right) < \left( \frac{e^{-1 + \beta}}{\beta ^{\beta}} \right)^{\nu} \leq \exp \left( (-1 + \beta - \beta \log \beta ) \alpha \log n \right).$$
When $\beta$ is a sufficiently small constant, $(-1+\beta - \beta \log \beta) \alpha$ is bounded below and away from $-1$, so
$$\Pr \left( \left |\Gamma^{[s]}_1(v) \right | < \beta \nu \right) = o \left( \frac{1}{n} \right).$$
Therefore, a.a.s., for every vertex $v \in X$, 
$$ \left |\Gamma^{[s]}_1(v) \right | \geq \beta \nu > \beta \log n.$$
Finally, by Lemma \ref{lemmaDelta}, the degree of each graph $G_i$ at $v$ is at most $\frac{2 c \log n}{\log \log n}$. Therefore, a.a.s, for each vertex $v \in X$,
$$ \left |\Gamma^{[s]\setminus C}_1(v) \right | > \beta \log n - |C| \left( \frac{2 c \log n}{\log \log n} \right) > \epsilon \log n,$$
for a sufficiently small positive constant $\epsilon>0$ depending only on $c$.
\end{proof}

We will use the value $\epsilon$ from Lemma \ref{lemFirstStep} throughout the rest of the paper.

With our final lemma of this section, we will show that when a breadth first search is carried out to compute a sphere $\Gamma^{[s] \setminus C} _k(v)$ for some positive integer $k$, the sets $\Gamma^{[s] \setminus C}_t(v)$ $(0 \leq t \leq k)$ at least double in size at each step of the search until a certain number of vertices are reached by the search. With this lemma, we can ensure that during all but the late steps of a breadth first search, at least half of the vertices reached by the search belong to the outer sphere. 

\begin{lemma}
\label{lemmaDouble}
It holds a.a.s.~that for every vertex $v \in X$, every set $C \subseteq [s]$ of size at most $ d+1$, and every $t$ with $0 \leq t \leq s - |C| - 1$ for which $\max_{1 \leq i \leq t} \left |\Gamma^{[s] \setminus C}_{i}(v) \right | \leq \frac{n}{10}$, we have 
$$\left |\Gamma^{[s] \setminus C}_{t+1}(v) \right | \geq 2 \left |\Gamma^{[s] \setminus C}_{t}(v) \right |.$$
\end{lemma}

\begin{proof}
We fix $v \in X$ and $C \subseteq [s]$. For every $t\ge0$ we let $V_t := \Gamma^{[s] \setminus C}_t(v)$ and $n_t := |V_t|$. 
As we have $n$ choices for $v$, fewer than $n^{d+1}$ choices for $C$, and fewer than $\log n$ steps $t$, it suffices to show that the inequality holds at each step with probability at least $1 - o(\exp(-(d+2) \log n - \log \log n))$. In fact, we will see that at each step, the probability of failure is only $\exp(-\gamma (\log n)^2)$, for some positive constant $\gamma> 0$.

The proof is by induction on $t$. When $t = 0$, the statement follows from Lemma \ref{lemFirstStep}. When $t = 1$, we show that a stronger bound holds. For a vertex $w \in V_1$ and a vertex $u \in X \setminus (V_0 \cup V_1)$, the probability that there exists an edge $uw$ in a graph $G_i$ for which some graph $G_j \neq G_i$ satisfies $vw \in E(G_j)$ is at least 
$$
   1 - (1-p)^{s- |C| - 1 } > \frac{1}{2}\, ps = \frac{c \log n}{2n}.
$$
Therefore, for each vertex $u \in X \setminus (V_0 \cup V_1)$, 
$$\Pr(u \in V_2) > 1 - \left( 1 - \frac{c \log n}{2n} \right)^{n_1}.$$
By Lemmas \ref{lemmaDelta} and \ref{lemFirstStep} and ($\star$), 
$$\epsilon \log n \leq n_1 \leq \frac{2c \log n}{\log \log n} \cdot s \leq 6c \left( \frac{\log n}{\log \log n} \right)^2.$$
Therefore, we may estimate that 
\begin{eqnarray*}
1 - \left( 1 - \frac{c \log n}{2n} \right)^{n_1} &>& n_1 \left( \frac{c \log n}{2n} \right) - n_1^2 \left( \frac{c \log n}{2n} \right)^2 \\
 & = & n_1 \cdot \frac{c \log n}{2n}\, (1 - o(1)) \\
 & > & n_1 \cdot \frac{c \log n}{3n} . 
\end{eqnarray*}
Hence, the expected number of vertices in $V_2$ is at least
$$ n_1 \cdot \frac{c \log n}{3n} \left(n - 6c \left( \frac{\log n}{\log \log n} \right)^2 -1 \right) > \frac{1}{4} \epsilon (\log n)^2.$$
Thus, by a Chernoff bound (Theorem \ref{thmSimpleChernoff}),
$$\Pr \left(n_2 < \frac{1}{8} \epsilon (\log n)^2 \right) < \exp \left( - \frac{1}{32} \epsilon (\log n)^2 \right) .$$
Hence, we may assume that $n_2 \geq \frac{1}{8} \epsilon (\log n)^2$, which is larger than $2n_1$~a.a.s.

Suppose now that $t \geq 2$ and $n_t\le\frac{n}{10}$. By the induction hypothesis, $n_0+n_1+\cdots +n_t \le n_t(1+\frac{1}{2}+\frac{1}{4}+\cdots+\frac{1}{2^t}) < \frac{n}{5}$. As $t \leq s - |C| - 1$, for each vertex $w \in V_t$, there exists at least one color $i \in [s] \setminus C$ such that our breadth first search may extend from $w$ using $G_i$. Using a similar argument as before, we have:
\begin{eqnarray}
\label{eqnGamma}
 \E[n_{t+1}] \geq \frac{4}{5}n \left(1-(1-p)^{n_t} \right) > \frac{4}{5}n \left(1-\exp (-n_t p) \right).
\end{eqnarray}
We can express the right-hand expression in (\ref{eqnGamma}) in the form $ n_t g(n, n_t)$, where  
$$g(n, n_t) = \frac{4n}{5 n_t} \left(1-\exp (-n_t p) \right).$$
Furthermore, by logarithmic differentiation,
\begin{eqnarray} 
\label{loggydif}
\frac{n_t}{g} \frac{\partial g}{\partial n_t} = -1 + \frac{n_t p \exp(-n_t p)}{1 - \exp(-n_t p)} .
\end{eqnarray}
Since $x < e^x - 1$ for $x > 0$, it follows that 
$$\frac{x e ^{-x}}{1 - e^{-x}} < 1$$
for $x > 0$. Therefore, $\frac{\partial g}{\partial n_t} < 0$, and $g(n,n_t)$ is minimized by increasing $n_t$. Thus, since $n_t \leq \frac{n}{10}$, 
$$\E \left [ n_{t+1} \right ] \geq n_t g \left(n,\frac{n}{10} \right) = 8 \left(1 - \exp \left(-\frac{np}{10} \right) \right)\cdot n_t > 4 n_t.$$
Hence, the expected value of $n_{t+1}$ is at least $4 n_t$. Then, by a Chernoff bound (Theorem \ref{thmSimpleChernoff}),
$$\Pr \left( n_{t+1} < 2 n_t \right) < \exp \left( - \frac{1}{2} n_t \right) < \exp \left(-\frac{1}{16} \epsilon (\log n)^2 \right) .$$
Therefore, the conclusion of the lemma a.a.s.~holds for all vertices $v \in X$ and all subsets $C \subseteq [s]$ of size at most $d+1$.
\end{proof}

\section{Proof of Theorem \ref{thmEasyLB}}

In this section, we will prove Theorem \ref{thmEasyLB}. Our strategy is to show that for an arbitrary vertex pair $u,v \in X$, $u$ and $v$ are rainbow connected with probability $1 - o \left( \frac{1}{n^2} \right)$, from which it will follow that $\mathcal G$ is a.a.s.~rainbow connected. 
Recall that by the bound (\ref{sBound}), we assume that $s \leq \frac{3 \log n }{\log \log n}$.
We will need the following fact.

\begin{lemma}
\label{lemmaLiftOff}
It a.a.s.~holds that for every vertex $u \in X$, there exists a value $t^* \leq d$
for which
\begin{equation}
\label{eqnLiftOff}
\left |\Gamma_{t^*}^{[s-1]}(u) \right | \geq \frac{1}{2}\epsilon (c \log n)^d.
\end{equation}
\end{lemma}

\begin{proof}
We consider a vertex $u \in X$.
To prove (\ref{eqnLiftOff}), we will perform a breadth first search as described in Section \ref{BFS}. As before, we will use the term \emph{unused vertices} to refer to those vertices that have not yet been reached by our search. In our breadth first search, if a set $\Gamma_t^{[s-1]} $ ever contains $(c \log n)^{d}$ vertices, then (\ref{eqnLiftOff}) is proven, along with the lemma. Otherwise, we may assume that at each step of our breadth first search, we have at least $n - d(c \log n)^{d}$ unused vertices.
For $1 \leq t \leq d-1$, we define the value
\begin{equation}
\label{eqn:delta1}
\delta_t = \left\{
  \begin{array}{ll}
    (\log n)^{-1/3}, & t=1,2; \\[0.5mm]
    ({\log n})^{-1}, & 3 \le t\le d-1.
  \end{array}
\right.
\end{equation}
Then, for $0 \leq t \leq d$, we define values $L_0 = 1$, $L_1 = \epsilon \log n$, and set
$$L_{t+1} = L_t \left( pn \right) ( s-t - 1) (1 - ps) \left(1 - \frac{d(c \log n)^{d}}{n} \right) (1 - \delta_t)$$
for $1 \leq t \leq d - 1$.
For $t \geq 1$, this gives us the following closed-form expression:
\begin{equation}
\label{eqnLLB}
L_t = (\epsilon \log n) \left( pn \right)^{t-1} (s-2)^{\underline{t-1}}(1 - ps)^{t-1} \left(1 - \frac{d(c \log n)^{d}}{n} \right)^{t-1} \prod_{i = 1}^{t-1}(1 - \delta_i).
\end{equation}
Using the inequality
$$(s-2)^{\underline{d-1}} > s^{d-1} \left(1 - \frac{d}{s} \right)^{d-1} > s^{d-1} \exp \left( - \frac{2d^2}{s} \right),$$
we can estimate
\begin{eqnarray*}
L_d &=& (\epsilon \log n) \left( pn \right)^{d-1} (s-2)^{\underline{d-1}}(1 - ps)^{d-1} \left(1 - \frac{d(c \log n)^{d}}{n} \right)^{d-1} \prod_{t = 1}^{d-1}(1 - \delta_t) \\
    &>& \epsilon (c \log n)^d \exp \left( - \frac{2d^2}{s} \right) (1 - ps)^{d-1} \left(1 - \frac{d(c \log n)^{d}}{n} \right)^{d-1} \prod_{t = 1}^{d-1}(1 - \delta_t)\\
    &>& \frac{1}{2} \epsilon (c \log n)^d. 
\end{eqnarray*}
Therefore, if we prove that
\begin{equation}
\label{eqnLBLBLB}
\left | \Gamma_{t}^{[s-1]}(u) \right | \geq L_t
\end{equation}
for $0\le t\le d$, then (\ref{eqnLiftOff}) will follow.

We will prove (\ref{eqnLBLBLB}) by induction on $t$.
We see that $\left | \Gamma_0^{[s-1]}(u) \right | = L_0 = 1$ by definition, and $\left | \Gamma_1^{[s-1]}(u) \right | \geq L_1 = \epsilon \log n$ by Lemma \ref{lemFirstStep}. Assuming that (\ref{eqnLBLBLB}) holds for values up to $t$, we estimate $\left | \Gamma_{t+1}^{[s-1]}(u) \right | $, for $t =1,\dots,d-1$. After $\Gamma_t^{[s-1]}(u)$ is computed, then for a vertex $u \in \Gamma_t^{[s-1]}(u)$ and an unused vertex $w \in X \setminus \bigcup_{j = 0}^t \Gamma_t^{[s-1]}(u)$, the probability that there exists an edge $vw$ in a graph $G_i$ for which $i$ does not appear on a given rainbow path from $u$ to $v$ is equal to $1 - (1-p)^{s-t-1} > (s-t-1)p -( s-t-1)^2 p^2$. Hence, if $\Gamma_t^{[s-1]}(u)$ has at least $L_t$ vertices, then
\begin{eqnarray}
\notag
\E \left [ \left | \Gamma_{t+1}^{[s-1]}(u) \right | \right ] \,>\, \mu_{t+1} &:=& L_t \left(( s-t-1)p - ( s-t-1)^2p^2 \right) (n - d (c \log n)^{d})\\
\notag
 & > & L_t \cdot np \cdot {(s - t-1)(1 - ps)} \left(1 - \frac{ d(c \log n)^{d} }{n} \right) \\
\label{eqnMuL}
 & = & L_{t+1} \, (1-\delta_t)^{-1}. 
\end{eqnarray}
For $t \in \{ 1,2\}$, we have the rough bound of $\mu_{t+1} > (\log n)^{7/4}$ by the induction hypothesis and (\ref{eqnLLB}), (\ref{eqnLBLBLB}), and (\ref{eqnMuL}). Hence, by a Chernoff bound (Theorem \ref{thmSimpleChernoff}), 
\begin{eqnarray*}
\Pr \left( \left | \Gamma_{t+1}^{[s-1]} (u) \right | < L_{t+1} \right) 
 &\le& \Pr \left(\left | \Gamma_{t+1}^{[s-1]} (u) \right | < \mu_{t+1}(1 - \delta_t) \right ) \\
 &\leq & \exp \left(-\frac{1}{2}\delta_t^2 \mu_{t+1} \right) \\
 &<& \exp \left(-\frac{1}{2} (\log n)^{13/12} \right)\\
 &=& o \left( \frac{1}{dn} \right).
\end{eqnarray*}
When $t \geq 3$, we have the rough bound of $\mu_{t+1} > (\log n)^{7/2}$ by the induction hypothesis and (\ref{eqnLLB}), (\ref{eqnLBLBLB}), and (\ref{eqnMuL}). Hence, by a Chernoff bound (Theorem \ref{thmSimpleChernoff}), 
\begin{eqnarray*}
\Pr \left(\left | \Gamma_{t+1}^{[s-1]}(u) \right | < L_{t+1} \right) & \le & \Pr \left(\left | \Gamma_{t+1}^{[s-1]} (u) \right | < \mu_{t+1}(1 - \delta_t) \right)\\
 & \leq& \exp\left(-\frac{1}{2}\delta_t^2 \mu_{t+1} \right) \\
 &< & \exp \left(-\frac{1}{2} (\log n)^{3/2} \right) \\
 & =& o \left( \frac{1}{d n} \right). 
\end{eqnarray*}
Therefore, with probability $1 - o \left(\frac{1}{n} \right)$, (\ref{eqnLBLBLB}) holds for all values $0 \leq t \leq d$, and the proof is complete.
\end{proof}

Now, we move to the main strategy. We consider a pair of distinct vertices $u,v \in X$. 
For each rainbow path $P$ with an endpoint at $u$ and with length at most $d$, we let $C(P)$ be the set of colors used in $E(P)$. We consider only rainbow paths with $C(P)\subseteq [s-1]$, and we define 
$$R_P := [s-1] \setminus C(P).$$ 
We also define
$$r := s - d - 1 \leq |R_P|.$$
We will consider spheres centered at $v$ obtained from a breadth first search using edges with colors in $R_P$. We write 
$$\xi = \frac{ n }{ (\log n)^{d-1}}$$
and we claim that $|\Gamma_t^{R_P}(v)| \geq \xi$ for some value $t \leq r$. We show this in the following lemma.

\begin{lemma}
\label{lemmaVSpheres}
Let $u,v \in X$ be a pair of distinct vertices, and let $P$ be a rainbow path of length at most $d$ with an endpoint at $u$. With probability $1 - o \left ( \frac{1}{n^4} \right )$, there exists a value $t\in[r]$ for which either $\Gamma_t^{R_P}(v)$ intersects $P$ or $\left |\Gamma_t^{R_P}(v) \right | \geq \xi$.
\end{lemma}

\begin{proof}
For $0 \leq t \leq r$, we write $V_t = \Gamma^{R_P}_t (v) $.
Similarly to Lemma \ref{lemmaLiftOff}, we will define values $L_t$ with the goal of showing that if $\left | V_t \right | < \xi$ for every $t \leq r$, then with high probability, $\left |V_t \right | \geq L_t$ for all values $0 \leq t \leq r$, which will ultimately give us a contradiction. In order to estimate $\left |V_t \right | $ for $0 \leq t \leq r$, we will carry out a breadth first search from $v$ as defined in Section \ref{BFS}. 
We will assume that this breadth first search never reaches a vertex of $P$, since otherwise, the lemma would be proven. 

For $1 \leq t \leq r-1$, we define $\delta_t$ as in (\ref{eqn:delta1}). 
Next, we define $\phi = \frac{1 }{\log n}$ and $\alpha_t = 1 - (r-t)p$. Finally, we define $L_0 = 1$, $L_1 = \frac{1}{2} \epsilon \log n$, and for $t \geq 2$, 
\begin{equation}
\label{eqnLLLB}
L_t = \frac{1}{2} \epsilon \log n  \left( \frac{c \log n}{s} \right)^{t-1} (r-1)^{\underline{t-1}} \, (1 - 2 \phi)^{2t-2}\, \prod_{i = 1}^{t-1} \alpha_i(1 - \delta_i).
\end{equation}
If there exists a value $k \leq r$ for which 
$\left | \bigcup_{t = 0}^k V_t \right |\geq \phi n$,
then there must exist a value $t \leq r$ for which $\left | V_t \right | \geq \frac{\phi n}{r} > \xi$. Therefore, we assume that we always have at least $n-\phi n$ unused vertices. 

\begin{claim}
\label{claimLLLL}
With probability $1 - o \left( \frac{1}{n^4} \right)$, for all values $0 \leq t \leq r$, $|V_t| \geq L_t$.
\end{claim}

\begin{proof}
We prove the claim by induction on $t$.
We let $\epsilon$ be the value from Lemma \ref{lemFirstStep}.
Clearly, $|\Gamma_0| = L_0 = 1$, and $|\Gamma_1| \geq L_1 = \epsilon \log n$ by Lemma \ref{lemFirstStep}. 

Now, for each value $t \geq 1$, we assume that $V_t $ is already computed, and we seek a lower bound for $V_{t+1}$. For notational simplicity, we write $m = |V_t |$.
 For an unused vertex $x \in X \setminus \bigcup_{j = 0}^t V_j$ and a vertex $y \in V_t$, the probability that $x$ is reached from $y$ in our breadth first search is at least 
$$1 - (1-p)^{r-t} > (r-t)p - (r-t)^2 p^2 = (r-t)p (1 - (r-t)p) = (r-t)p\alpha_t.$$
Therefore, using the assumption that $m < \xi$, the probability that $x$ is reached from at least one vertex of $V_t $ is at least 
\begin{eqnarray*}
1- (1 - (r-t)p\alpha_t)^{m} &>& (r-t) p \alpha_t m - (r-t)^2 p^2 \alpha_t^2 m^2 \\
 & \geq & (r-t)p \alpha_t m ( 1 - sp \xi) \\
 & = & (r-t)p \alpha_t m \left( 1 - \frac{ c \log n}{n} \cdot \frac{ n }{ (\log n)^{d-1}} \right)\\
 &>& (r-t)p \alpha_t m ( 1 - 2 \phi).
\end{eqnarray*}
Hence, 
\begin{eqnarray*}
\mu_{t+1}:= \E \left [ |V_{t+1}| \right ] & > & n(1 -  \phi) \left( 1- (1 - (r-t)p\alpha_t)^{m} \right) \\
 & > & n( 1 - 2 \phi)^2 p \alpha_t (r - t) m \\
 &\geq& L_t \cdot n( 1 - 2 \phi) ^2 p \alpha_t (r - t)  \\
 & = &L_t  \cdot \frac{c \log n}{s} \cdot  (1 - 2 \phi)^2 \alpha_t ({r - t}) .
\end{eqnarray*}
When $t \in \{1,2\}$, the equation (\ref{eqnLLLB}) and the induction hypothesis give us the rough bound $\mu_{t+1} > (\log n)^{11/6}$. Therefore, by a Chernoff bound (Theorem \ref{thmSimpleChernoff}),
\begin{eqnarray*}
\Pr \left( |V_{t+1}| < L_{t+1} \right)
 &=& \Pr \left( |V_{t+1}|  < \mu_{t+1} (1 - \delta_t) \right) \\
 &\leq & \exp \left(-\frac{1}{2} \delta_t^2 \mu_t \right)\\
 & < & \exp \left(-\frac{1}{2} (\log n)^{7/6} \right) \\
 & = & o\left( \frac{1}{n^5} \right).
\end{eqnarray*}
 When $t \geq 3$, the equation (\ref{eqnLLLB}) and the induction hypothesis give us the rough bound $\mu_t > (\log n)^{7/2}$. Therefore, 
\begin{eqnarray*}
\Pr \left( |V_{t+1}| < L_{t+1} \right) 
 &=& \Pr \left( |V_{t+1}| < \mu_t(1 - \delta_t) \right) \\
 & \leq& \exp \left( -\frac{1}{2} \delta_t^2 \mu_t \right) \\
 & <& \exp \left( -\frac{1}{2}(\log n)^{3/2} \right ) \\
 & =& o \left( \frac{1}{n^5} \right). 
\end{eqnarray*}
Putting these bounds together, the probability that $|V_t| < L_t$ for some $t$ is at most 
$o \left( \frac{1}{n^5} \right) + r\, o \left( \frac{1}{n^5} \right) = o \left( \frac{1}{n^4} \right).$  This completes the proof.
\end{proof}

Now, we apply Claim \ref{claimLLLL} with $t = r$ and obtain the following bound:
$$
|V_r| \geq \frac{1}{2} \epsilon \log n ( pn )^{r-1} (r-1)! (1 - \phi)^{2r-2} \prod_{i = 1}^{r-1} \alpha_i (1 - \delta_i) .
$$
By Stirling's approximation, $(r-1)! = \frac{1}{r} \cdot r! \geq r^{r-1} e^{-r} \sqrt{2 \pi r}$. Therefore,
\begin{eqnarray*}
 |V_r| 
 &\geq& \frac{1}{2} \epsilon \log n \left( pnr \right)^{r-1} e^{-r} \sqrt{2\pi r} (1 - \phi)^{2r-2} \prod_{i = 1}^{r-1} \alpha_i (1 - \delta_i) \\
 &\geq& 
 \frac{1}{2} \epsilon \log n ( c \log n)^{r-1} (r/s)^{r-1} e^{-r} \sqrt{2 \pi r} (1 - \phi)^{2r-2} \prod_{i = 1}^{r-1} \alpha_i (1 - \delta_i) \\
 &=& \exp \bigg (r \log \log n + (r-1) \left(\log c + \log\left( \frac{r}{s} \right) + 2\log (1- \phi)\right) - r + \frac{1}{2} \log r + O(1) \bigg ).
\end{eqnarray*}
We recall that $\phi = \frac{1}{\log n}$,
and then using the inequality $-\log(1-x) < 2x$ for small $x$, it follows that
$$\left | (r-1)\log(1-\phi) \right | < 2 (r-1) \phi = O(1).$$
Furthermore, 
$$\left | (r-1) \log \left( \frac{r}{s} \right) \right| = \left | (r-1) \log \left(1 - \frac{d+1}{s} \right) \right | < 2 (r-1)\cdot \frac{d+1}{s} = O(d).$$
Hence, we may write more simply:
\begin{eqnarray*}
 |V_{r}| 
 &\geq& \exp \big (r \log \log n + (r-1) \log c - r + \frac{1}{2} \log r + O(d) \big ) \\
 &\geq& \exp \left( (s -d - \tfrac{1}{2} ) \log \log n + s (\log c - 1) + O(d) \right).
\end{eqnarray*}
Then,
\begin{equation}
 \frac {|V_r|}{n / (\log n)^{d-1}} 
 \geq \exp \big( s(\log \log n + \log c - 1) - \tfrac{3}{2} \log \log n - \log n + O(d) \big). \label{eq:end1}
\end{equation} 

By the hypothesis of Theorem \ref{thmEasyLB} we have that
\begin{equation*}
s \geq \frac{\log n}{\log c - 1 + \log \log n} + \frac{3}{2} + \frac{2 \sqrt{ \log \log \log n }}{ {\log \log n}}.
\end{equation*}
This assumption and (\ref{eq:end1}) imply the following:
\begin{eqnarray*}
 \frac {|V_r|}{n / (\log n)^{d-1}} 
 &\geq& \exp \bigg( \frac{3}{2}(\log c - 1) + \frac{2\sqrt{\log\log\log n}}{\log\log n} (\log\log n + \log c - 1) + O(d) \bigg) \\
 &\geq& \exp \big( \sqrt{\log\log\log n} + O(d) \big) \\
 & > & 1.
\end{eqnarray*}
Thus, we conclude that
$$|V_r| \geq \frac{ n }{ (\log n)^{d-1}} = \xi.$$
Therefore, the assumption that $|V_t| < \xi$ for every value $0 \leq t \leq r$ is contradicted with probability $1 - o \left ( \frac{1}{n^4} \right )$, and thus the lemma is proven.
\end{proof}

Now, using Lemmas \ref{lemmaLiftOff} and \ref{lemmaVSpheres}, we are ready to prove Theorem \ref{thmEasyLB}. We choose a pair of distinct vertices $u,v \in X$. If we are able to show that there exists a rainbow path with endpoints $u$ and $v$ with probability $1 - o \left ( \frac{1}{n^2} \right )$, then this will show that $\mathcal G$ is a.a.s.~rainbow connected, and Theorem \ref{thmEasyLB} will be proven. By Lemma \ref{lemmaVSpheres}, for each rainbow path $P$ with an endpoint at $u$ and with length at most $d$, it holds with probability $1 - o \left ( \frac{1}{n^4} \right )$ that there exists a value $t_P$ for which $\Gamma_{t_P}^{R_P}(v)$ intersects $P$ or contains at least $\xi$ vertices. Using Lemma \ref{lemmaDelta}, the number of rainbow paths with an endpoint at $u$ and of length at most $d$ is bounded above by $ds^d \left ( \frac{2c \log n}{\log \log n} \right )^d < n$. Therefore, with probability $1 - o\left ( \frac{1}{n^3} \right )$, it holds for every rainbow path $P$ with an endpoint at $u$ and of length at most $d$ that there exists a value $t_P$ for which $\Gamma_{t_P}^{R_P}(v)$ intersects $P$ or contains at least $\xi$ vertices. If one of the rainbow paths $P$ is intersected by $\Gamma_{t_P}^{R_P}(v)$, then clearly $u$ and $v$ are connected by a rainbow path; therefore, we may assume that for every $P$, we can find a sphere $\Gamma_{t_P}^{R_P}$ with at least $\xi$ vertices.

Now, we define a set $E_{uv}$ of vertex pairs as follows. If there exists a rainbow path $P$ of length at most $d$ with endpoints $u$ and $w$, and if $x \in \Gamma_{t_P}^{R_P}(v)$, then we add the pair $\{w,x\}$ to $E_{uv}$. We estimate the number of vertex pairs in $E_{uv}$. By Lemma \ref{lemmaLiftOff}, there exists a set $\mathcal P$ of at least $\frac{1}{2} \epsilon (c \log n)^d$ rainbow paths of length at most $d$, each with $u$ as an endpoint, and all with distinct second endpoints $w$ in some sphere $\Gamma_{t^*}^{[s-1]}(u) $. Furthermore, for each path $P \in \mathcal P$, all of the vertices $x \in \Gamma_{t_P}^{R_P}(v)$ are distinct. Therefore, each pair $\{w,x\}$ with $w \in \Gamma_{t^*}^{[s-1]}(u) $ can be added at most twice to $E_{uv}$, once with $w$ as an endpoint of a path $P \in \mathcal P$ and with $x \in \Gamma_{t_P}^{R_P}(v)$, and once with $x$ as an endpoint of a path $P \in \mathcal P$ and with $w \in \Gamma_{t_P}^{R_P}(v)$. By Lemma \ref{lemmaVSpheres}, for each path $P \in \mathcal P$, at least $\xi$ pairs are added to $E_{uv}$, giving $E_{uv}$ a total of at least 
\[ \frac{1}{4} \epsilon (c \log n)^d \xi > 4 \log n / p\]
distinct pairs. 
With probability at least 
$$1-(1 - p)^{4 \log n / p} = 1 - o \left(\frac{1}{n^2} \right),$$
$G_s$ has an edge at some pair in $E_{uv}$, giving us a rainbow path between $u$ and $v$ with probability $1 - o \left ( \frac{1}{n^2} \right )$. Since the number of vertex pairs $u,v \in X$ is less than $n^2$, it follows that every vertex pair in $X$ is connected by a rainbow path with probability $1 - o(1)$. This completes the proof of Theorem \ref{thmEasyLB}.

\section{An alternative colorful random setting}
\label{sec:differentSetting}

The results in the previous sections all belong to the setting in which a randomly edge-colored graph is obtained by taking the union of many monochromatically edge-colored graphs on a common vertex set. As mentioned in the introduction, a different random model is obtained by taking a single random graph $G$, each of whose edges is given a single color from the set $[s]$ uniformly at random. These two settings are not equivalent; for example, in the first setting, a single vertex-pair may have edges of multiple colors, while in the second setting, this is not possible. Nevertheless, these two models are similar in many ways, and in fact, Theorems \ref{thmUB} and \ref{thmEasyLB} have direct counterparts in this new random setting.

\begin{theorem}
\label{thm:newUB}
Let $G$ be a graph taken randomly from $G(n,p)$, with $p = \frac{c \log n}{n}$, where $c>1$ is a constant. Let each edge of $G$ be given a color from $[s]$ uniformly at random.
If $$s \leq \frac{\log n}{\log c - 1 + \log \log n} - \frac{1}{2} + \frac{\log \log \log n}{3 \log \log n},$$
then a.a.s.~$G$ is not rainbow connected. 
\end{theorem}

\begin{theorem}
\label{thm:newLB}
Let $G$ be a graph taken randomly from $G(n,p)$, with $p = \frac{c \log n}{n}$, where $c>1$ is a constant. Let each edge of $G$ be given a color from $[s]$ uniformly at random.
If 
$$s \geq \frac{\log n}{\log c - 1 + \log \log n} + \frac{3}{2} + \frac{2 \sqrt{ \log \log \log n }}{ {\log \log n}} ,$$
then a.a.s.~$G$ is rainbow connected. 
\end{theorem}

We sketch some of the details that one might use to prove these theorems in the new setting. In order to prove Theorem \ref{thm:newUB}, the First Moment Method calculation used in the proof of Theorem \ref{thmUB} can be copied exactly. In order to prove Theorem \ref{thm:newLB}, the techniques and lemmas used to prove Theorem \ref{thmEasyLB} can be followed very closely. First, the inequality (\ref{sBound}) can be obtained using the same method as in Lemma \ref{lemmaSmall}. Then, the same breadth first search method can be employed, and we again can use Chernoff bounds in order to estimate the sizes of spheres. In fact, when we compute lower bounds for the sizes of spheres, we will often obtain better estimates. The reason for this is that in the original setting, we estimate the probability that a vertex pair contains an edge belonging to a color set $C$ as 
\[1-(1-p)^{|C|} > p|C| - p^2 |C|^2 = (1- o(1)) \frac{c |C| \log n}{sn}.\]
However, in the new setting, we can write this probability exactly as 
\[\frac{p|C|}{s} = \frac{c  |C|\log n}{sn},\] which gives us approximately the same probability, but with no error term. Therefore, the lower bounds for sphere sizes that we use in the original setting can be translated into the new setting without any changes. Finally, our last technique of finding an edge of color $s$ between the endpoints of two rainbow paths can be used without any modifications. Therefore, our overall proof structure can be used in both settings.

\raggedright
\bibliographystyle{abbrv}
\bibliography{RCBib}
\end{document}